\documentclass[12pt,a4paper]{article}
\usepackage{latexsym,amsfonts,amsmath,graphics,amsthm}
\usepackage{epsfig}

\newtheorem{theorem}{Theorem}
\newtheorem{lemma}{Lemma}

\newtheorem{proposition}{Proposition}

\newtheorem{definition}{Definition}

\newtheorem{remark}{Remark}

\newcommand{\To}{\rightarrow}

\newcommand{\bsa}{\boldsymbol{a}}

\newcommand{\bsk}{\boldsymbol{k}}
\newcommand{\bsh}{\boldsymbol{h}}

\newcommand{\bsl}{\boldsymbol{l}}

\newcommand{\bsx}{\boldsymbol{x}}

\newcommand{\bsi}{\boldsymbol{i}}

\newcommand{\bsy}{\boldsymbol{y}}

\newcommand{\bsnu}{\boldsymbol{\nu}}

\newcommand{\C}{{\cal C}}
\newcommand{\N}{{\cal N}}

\newcommand{\wal}{{\rm wal}}

\newcommand{\icomp}{\mathtt{i}}
\newcommand{\bszero}{\boldsymbol{0}}

\newcommand{\rd}{\,\mathrm{d}}
\newcommand{\Field}{\mathbb{F}}

\newcommand{\Natural}{\mathbb{N}}

\makeatletter
\newcommand{\rdots}{\mathinner{\mkern1mu\lower-1\p@\vbox{\kern7\p@\hbox{.}}
\mkern2mu \raise4\p@\hbox{.}\mkern2mu\raise7\p@\hbox{.}\mkern1mu}}
\makeatother

\allowdisplaybreaks

\begin{document}

\title{\scshape A higher order Blokh-Zyablov propagation rule for higher order nets}

\author{Josef Dick\thanks{J. Dick gratefully acknowledges the support of the Australian Research Council.} 
 and Peter Kritzer\thanks{P. Kritzer is supported by the Austrian Science Fund (FWF), Project P23389-N18.}}

\date{\today}
\maketitle

\begin{abstract}
Higher order nets were introduced by Dick as a generalisation of
classical $(t,m,s)$-nets, which are point sets frequently used in
quasi-Monte Carlo integration algorithms. Essential tools in
finding such point sets of high quality are propagation rules,
which make it possible to generate new higher order nets from
existing higher order nets and even classical $(t,m,s)$-nets. Such propagation rules for higher order nets were first considered by the authors in \cite{DK08} and
further developed in \cite{BDP11}. In \cite{BZ74} Blokh and
Zyablov established a very general propagation rule for linear
codes. This propagation rule has been extended to $(t,m,s)$-nets
by Sch\"urer and Schmid in \cite{mint}. In this paper we show that
this propagation rule can also be extended to higher order nets. Examples indicate that this propagation rule yields new higher order nets with significantly higher quality.
\end{abstract}



\section{Introduction}

Quasi-Monte Carlo (QMC) rules are quadrature rules for approximating a high-dimensional integral
$\int_{[0,1]^s} f(\bsx)\rd\bsx$ by the average
of certain function values, $\frac{1}{N}\sum_{i=0}^{N-1}f(\bsx_i)$, where
$\bsx_0,\bsx_1,\ldots,\bsx_{N-1}$ are deterministically chosen
integration nodes. The crucial question is how to choose the
points $\bsx_i$ in order to obtain low integration errors. For
this reason, several different classes of point sets and their
properties in relation to integration problems have been
studied over the past decades (see, e.g., the monographs
\cite{DPCam, niesiam, SJ94} for introductions to this topic). Note that by a point set we always mean a
multiset, i.e., points are allowed to occur repeatedly.
One prominent class of point sets commonly used in QMC algorithms
are $(t,m,s)$-nets as introduced by Niederreiter (cf. \cite{N87,
niesiam, NiedNets}).  It is
known that $(t,m,s)$-nets can yield the optimal order of
convergence (up to powers of the logarithm of the total number of
points) of the integration error in QMC algorithms for functions
of finite variation in the sense of Hardy and Krause (see again
\cite{N87, niesiam, NiedNets}).

In \cite{D07, D08}, Dick generalised Niederreiter's $(t,m,s)$-nets to
so-called higher order nets which have the convenient property
that the corresponding QMC integration algorithms yield higher
order convergence of the error for smoother functions. To be more
precise, in \cite{D08} digital higher order nets were introduced,
the construction of which, as it is also the case for
Niederreiter's digital nets, is based on linear algebra over
finite fields. In what follows, let $\mathbb{F}_q$ be the finite
field of prime-power order $q$. Moreover, we denote by $\Natural$ and $\Natural_0$ 
the set of positive integers and nonnegative integers, respectively.

\begin{definition}\label{defmatricespoints}
Let $q$ be a prime power and let $n,m,s\in\Natural$. Let
$C_1,\ldots,C_s$ be $n\times m$ matrices over the finite field
$\Field_q$ of order $q$. We construct $q^m$ points in $[0,1)^s$ in
the following way: For $0\le h< q^m$ let $h=h_0+h_1 q +\cdots
+h_{m-1}q^{m-1}$ be the base $q$ representation of $h$. Consider
an arbitrary but fixed bijection
$\eta:\{0,1,\ldots,q-1\}\To\Field_q$ where $\eta(0)$ is the zero
element in $\Field_q$. Identify $h$ with the vector
$\bsh:=(\eta(h_0),\ldots,\eta(h_{m-1}))\in\Field_q^m$. For $1\le
j\le s$, we multiply the matrix $C_j$ by $\bsh$,
\[C_j\cdot\bsh^\top =:(y_{j,1}(h),\ldots,y_{j,n}(h))\in\Field_q,\]
and set
\[x_{h}^{(j)}:=\frac{\eta^{-1}(y_{j,1}(h))}{q}+\cdots+\frac{\eta^{-1}(y_{j,n}(h))}{q^n}.\]
Finally, set
$\bsx_{h}:=\left(x_{h}^{(1)},\ldots,x_{h}^{(s)}\right)$. The point
set consisting of the points $\bsx_0,\ldots,\bsx_{q^m-1}$ is
called a digital net over $\Field_q$. The matrices
$C_1,\ldots,C_s$ are called the generating matrices of the digital
net.
\end{definition}

As it can be seen from Definition \ref{defmatricespoints}, the
properties of the points of a digital net (such as, e.g., their
distribution in the unit cube) are determined by properties of the
generating matrices $C_1,\ldots,C_s$. These properties are, in the
currently most general form of digital nets as introduced in
\cite{D08}, described by additional parameters $t,\alpha,\beta$,
which is why those nets are referred to as
$(t,\alpha,\beta,n\times m,s)$-nets. The exact roles of the
parameters $t$, $\alpha$, and $\beta$ are outlined in the following
definition.

\begin{definition}\label{defhonet}
Let $n,m,\alpha\in\Natural$, let $0<\beta\le
\min(1,\alpha m/n)$ be a real number. Let $\Field_q$ be the finite
field of prime power order $q$ and let
$C_1,\ldots,C_s\in\Field_q^{n\times m}$ with
$C_j=\left(\vec{c}_{j,1},\ldots,\vec{c}_{j,n}\right)^{\top}$. The
digital net with generating matrices $C_1,\ldots, C_s$ is called a
digital $(t,\alpha,\beta,n\times m,s)$-net for a natural number
$t$, $0\le t\le \beta n$, if the following condition is satisfied.
For each choice of $1\le i_{j,\nu_j}<\cdots<i_{j,1}\le n$, where
$\nu_j\ge 0$ for $j=1,\ldots,s$, with
\begin{equation}\label{eqcondidef}
 i_{1,1}+\cdots +i_{1,\min\{\nu_1,\alpha\}}+\cdots+i_{s,1}+\cdots+i_{s,\min\{\nu_s,\alpha\}}\le\beta n -t
\end{equation}
the vectors
\[
 \vec{c}_{1,i_{1,\nu_1}},\ldots,\vec{c}_{1,i_{1,1}},\ldots,\vec{c}_{s,i_{s,\nu_s}},\ldots,\vec{c}_{s,i_{s,1}}
\]
are linearly independent over $\Field_q$.

If $t$ is the smallest non-negative integer such that the digital
net generated by $C_1,\ldots, C_s$ is a digital
$(t,\alpha,\beta,n\times m,s)$-net, then we call the digital net a
strict digital $(t,\alpha,\beta, n \times m,s)$-net.
\end{definition}

Note that Definition \ref{defhonet} implies that $t$ must be
chosen such that $\nu_1+\cdots+\nu_s\le m$ holds whenever
\eqref{eqcondidef} is satisfied. (Note that $\nu_j \le i_{j,1}$.) 

The definition of classical digital $(t,m,s)$-nets is obtained by
choosing $\alpha = \beta = 1$ and $m = n$ in
Definition~\ref{defhonet}.

Digital higher order nets are a subclass of general higher order
nets which were introduced in \cite{DickBalGeomprop}.
While we denote digital higher order nets as $(t,\alpha,\beta,n\times m,s)$-nets, by which
we emphasise the role of the generating matrices in the construction, general higher order nets are denoted as $(t,\alpha,\beta,n,m,s)$-nets.

We give the definitions and some properties of
$(t,\alpha,\beta,n,m,s)$-nets in base $b$. We follow \cite{BDP11}
in our presentation.

Let $n,s \ge 1$, $b \ge 2$ be integers. For $\bsnu =
(\nu_1,\dots,\nu_s)\in \left\{ 0,\dots,n \right\}^s$ let $\vert
\bsnu \vert_1 = \sum^s_{j=1} \nu_j$ and define $\bsi_{\bsnu}=
(i_{1,1}, \dots, i_{1,\nu_1}, \dots, i_{s,1}, \dots,i_{s,\nu_s})$
with integers $1 \leq i_{j,\nu_j} < \cdots < i_{j,1} \leq n$ in
case $\nu_j >0$ and $\left\{i_{j,1},\dots,i_{j,\nu_j}
\right\}=\emptyset$ in case $\nu_j=0$, for $j=1,\dots,s$. For
given $\bsnu$ and $\bsi_{\bsnu}$ let $\bsa_{\bsnu} \in \left\{ 0,
\dots, b-1\right\}^{\vert \bsnu \vert_1}$, which we write as
$\bsa_{\bsnu} = (a_{1,i_{1,1}}, \dots, a_{1,i_{1,\nu_1}}, \dots ,
a_{s,i_{s,1}}, \dots, a_{s,i_{s, \nu_s}})$.

A \emph{generalised elementary interval in base} $b$ is a subset
of $[0,1)^s$ of the form
\begin{eqnarray*}
J( \bsi_{\bsnu}, \bsa_{\bsnu} )  = \prod^s_{j=1}
\bigcup^{b-1}_{\stackrel{a_{j,l}=0}{l \in \{1,\dots,n \} \setminus
\{i_{j,1}, \ldots, i_{j,\nu_j}\} } } \left[ \frac{a_{j,1}}{b} +
\dots + \frac{a_{j,n}}{b^n}, \frac{a_{j,1}}{b} + \dots +
\frac{a_{j,n}}{b^n} + \frac{1}{b^n} \right).
\end{eqnarray*}


From \cite[Lemmas 1 and 2]{DickBalGeomprop} it is known that for
$\bsnu \in \left\{ 0,\dots,n \right\}^s$ and $\bsi_{\bsnu}$
defined as above and fixed, the generalised elementary intervals
$J(\bsi_{\bsnu}, \bsa_{\bsnu} )$ where $\bsa_{ \bsnu}$ ranges over
all elements from the set $\left\{ 0,\dots, b-1 \right\}^{ \vert
\bsnu \vert_1 }$
form a partition of $[0,1)^s$ and the volume of $J( \bsi_{\bsnu}, \bsa_{\bsnu} )$ is $b^{- \vert \bsnu \vert_1 }$.\\

We can now give the definition of $(t,\alpha,\beta,n,m,s)$-nets
based on \cite[Definition~4]{DickBalGeomprop}.

\begin{definition}\rm \label{deftalphabetanets} Let $n$, $m$, $s$, $\alpha \geq 1$ be natural numbers, let
$0 < \beta \leq \min(1,\alpha m / n)$ be a real number, and let $0
\leq t \leq \beta n$ be an integer. Let $b \geq 2$ be an integer
and $P = \{\bsx_0,\dots,\bsx_{b^m -1}\}$ be a multiset in
$[0,1)^s$. We say that $P$ is a $(t,\alpha,\beta,n,m,s)$-{\it net
in base} $b$, if for all integers $1 \leq i_{j,\nu_j} < \dots <
i_{j,1}$, where $0 \le \nu_j \leq n$, with
\begin{displaymath} \sum^s_{j=1} \sum^{\min(\nu_j,\alpha)}_{l=1} i_{j,l} \leq \beta n -t, \end{displaymath}
where for $\nu_j=0$ we set the empty sum $\sum^0_{l=1} i_{j,l}
=0$, the generalised elementary interval $J(\bsi_{\bsnu},
\bsa_{\bsnu})$ contains exactly $b^{m-\vert \bsnu \vert_1}$ points
of $P$ for each $\bsa_{\bsnu} \in \left\{ 0,\dots,b-1
\right\}^{\vert \bsnu \vert_1 }$.

A $(t,\alpha,\beta,n,m,s)$-net in base $b$ is called a {\it
strict} $(t,\alpha,\beta,n,m,s)$-{\it net in base} $b$, if it is
not a $(u,\alpha,\beta,n,m,s)$-net in base $b$ with $u<t$.

Informally we refer to $(t,\alpha,\beta, n,m, s)$-nets as higher
order nets.
\end{definition}

Note that in the definition above the specific order of elements
of a multiset is not important. The parameter $t$ is often
referred to as the {\it quality parameter} of the net. By the {\it
strength} of the net one means the quantity $\beta n -t$.

The advantage of the more general concept due to
\cite{DickBalGeomprop} (in comparison to classical $(t,m,s)$-nets)
is that $(t,\alpha,\beta,n,m,s)$-nets in base $b$ can exploit the
smoothness $\alpha$ of a function $f$ (which is not the case for
the classical concepts of $(t,m,s)$-nets and $(t,s)$-sequences).
More precisely, we have the following result from \cite{BDP09} (and \cite{D07, D08} for the digital case).
\begin{proposition}\label{thm_wcebound2}
Let $\{\bsx_0,\ldots, \bsx_{b^m-1}\}$ be a $(t,\alpha,\beta,n, m,
s)$-net in base $b$. Let $f:[0,1]^s \rightarrow \mathbb{R}$ have
mixed partial derivatives up to order $\alpha \ge 2$ in each
variable which are square integrable. Then $$\left| \int_{[0,1]^s}
f(\bsx) \rd \bsx  - \frac{1}{b^m} \sum_{h=0}^{b^m-1}
f(\bsx_h)\right| = {\cal O}\left(b^{- (1-1/\alpha) (\beta n - t)}
(\beta n - t)^{\alpha s}\right).$$ If $\{\bsx_0,\ldots, \bsx_{b^m-1}\}$ is a digital $(t,\alpha,\beta,n, m, s)$-net in base $b$, then $$\left| \int_{[0,1]^s}
f(\bsx) \rd \bsx  - \frac{1}{b^m} \sum_{h=0}^{b^m-1}
f(\bsx_h)\right| = {\cal O}\left(b^{- (\beta n - t)}
(\beta n - t)^{\alpha s}\right).$$
\end{proposition}
Additionally, the following results are known. If $\alpha =  \beta
= 1$ and $n = m$, then the integration error is of order ${\cal
O}(b^{-m+t} m^s)$, see \cite{niesiam}. 

Proposition~\ref{thm_wcebound2} underlines the importance of knowing
explicit constructions of higher order nets with a large value of
$\beta n - t$.

In a series of papers, see for example
\cite{bieredelschmid,NO,niexing98,SS09}, so-called propagation
rules for classical $(t,m,s)$-nets were introduced, which allow
one to construct new digital nets from known ones and thereby
improve on the parameters, in particular on the strength, of those
nets. That such constructions are very useful can be seen in
\cite{mint}, where the best known parameters of classical
$(t,m,s)$-nets are listed. A first step in establishing and using
propagation rules for digital higher order nets (i.e.,
$(t,\alpha,\beta,n\times m,s)$-nets) was made in the paper
\cite{DK08}, where a series of propagation rules for such point sets were discussed. In the later paper \cite{BDP11},
these propagation rules were extended to the case where the
underlying nets need no more be digital, i.e., the rules apply to 
the more general class of $(t,\alpha,\beta,n,m,s)$-nets.

In this paper, it is our aim to show a further propagation rule
for higher order nets, which is a modification of a propagation
rule for linear codes by Blokh and Zyablov \cite{BZ74}. Our main
result, which will be prepared in Section \ref{secbzconstr} and stated in Section \ref{sec-blokh}, is an extension of 
a Blokh-Zyablov propagation rule for classical $(t,m,s)$-nets by Sch\"urer and Schmid in \cite{SS09}. Before
(Section \ref{secdual}), we are going to review duality theory of
higher order nets, as it was introduced in \cite{DK08}  and
\cite{BDP11}. Duality theory will be the main tool in proving our
new propagation rule for higher order nets. Finally, we outline
exemplary numerical results in Section \ref{secnum}, which are
going to highlight the strength of our new propagation rule.

\section{Duality theory}\label{secdual}

In this section we review the basics of duality
theory for higher order nets, as this is going to be the
key ingredient for showing our main results. Duality theory, as
introduced for classical $(t,m,s)$-nets by Niederreiter and Pirsic
\cite{NP}, is a helpful tool in the analysis and construction of
digital nets. In \cite{DK08} it was extended to digital higher
order nets, and it was further extended to the non-digital case in
\cite{BDP11}. We keep close to \cite{BDP11} and \cite{DK08} in the
following two sections.

\subsection{Duality theory for digital higher order nets}

In \cite{DK08}, the dual of the row space of a digital
$(t,\alpha,\beta, n\times m, s)$-net was introduced. Given the
generating matrices $C_1,\ldots,C_s$ of a digital higher order
net, let
\begin{equation*}
C = (C_1^\top \mid \cdots \mid C_s^\top) \in \Field_q^{m \times
sn}.
\end{equation*}
The points of a digital net are obtained from the linear subspace
$\C$ of $\mathbb{F}_q^{ns}$ of dimension at most
$m$, given as the row space of $C$ by Definition~\ref{defmatricespoints}. The dual space of $\C$ is given by
\begin{equation*}
\mathcal{N} = \{\mathbf{x} \in \mathbb{F}_q^{ns}: \mathbf{x} \cdot
\mathbf{y} = 0 \in \mathbb{F}_q \mbox{ for all } \mathbf{y} \in
\mathcal{P}\}, 
\end{equation*}
i.e., $\N$ is the null space of $C$.

Let $\mathbf{A} = (\mathbf{a}_1,\ldots, \mathbf{a}_s) \in
\mathbb{F}_q^{ns}$, where $\mathbf{a}_j = (a_{j,1},\ldots,
a_{j,n}) \in \mathbb{F}_q^n$. For $1 \le j \le s$ with
$\mathbf{a}_j \neq \mathbf{0}$ let $1 \le \nu_j \le n$ denote the
number of nonzero elements of $\mathbf{a}_j$ and let $1 \le
i_{j,\nu_j} < \cdots < i_{j,1} \le n$ denote the indices of the
nonzero elements $a_{j,i_{j,\nu_j}},\ldots, a_{j,i_{j,1}} \in
\mathbb{F}_q^\ast = \mathbb{F}_q\setminus \{0\}$ of $\mathbf{a}_j$
(thus, $a_{j,l} = 0$ for $l \in \{1,\ldots, n\} \setminus
\{i_{j,1},\ldots, i_{j,\nu_j}\}$). Let $\alpha \ge 1$ be an
integer. Then we define
\begin{equation*}
\mu_{\alpha,n}(\mathbf{a}_j) = \left\{\begin{array}{ll} 0 &
\mbox{for } \mathbf{a}_j = \mathbf{0}, \\
\sum_{k=1}^{\min(\alpha,\nu_j)} i_{j,k} & \mbox{otherwise},
\end{array} \right.
\end{equation*}
and
\begin{equation*}
\mu_{\alpha,n}(\mathbf{A}) = \sum_{j=1}^s
\mu_{\alpha,n}(\mathbf{a}_j).
\end{equation*}

Let $\mathcal{N} \subseteq \mathbb{F}_q^{ns}$ and let
$|\mathcal{N}|$ denote the number of elements of $\mathcal{N}$.
For $\mathbf{A}, \mathbf{B} \in \mathcal{N}$ we define the
distance
\begin{equation*}
d_{\alpha,n}(\mathbf{A},\mathbf{B}) = \mu_{\alpha, n}(\mathbf{A} -
\mathbf{B}).
\end{equation*}
Further we define
\begin{equation*}
\delta_{\alpha,n}(\mathcal{N}) = \min_{\substack{\mathbf{A}, \mathbf{B} \in
\mathcal{N}\\ \mathbf{A} \neq \mathbf{B}}} d_{\alpha,
n}(\mathbf{A},\mathbf{B}).
\end{equation*}
We always have $\delta_{\alpha,n}(\N) \ge 1$ and
$\delta_{\alpha,n}(\N) \ge \delta_{\alpha',n}(\N)$ for $\alpha \ge
\alpha' \ge 1$.

The following definition and result was first shown
in \cite{DK08} and is a generalisation of \cite{NP}.

\begin{proposition}\label{thm1}
The matrices $C_1,\ldots, C_s \in \mathbb{F}_q^{n \times m}$
generate a digital $(t,\alpha,\beta,n \times m,s)$-net over
$\mathbb{F}_q$ if and only if
\begin{equation*}
\delta_{\alpha,n}(\N) \ge \beta n - t + 1,
\end{equation*}
where $\mathcal{N}$ is the dual space of the row space
$\mathcal{C}$. If $C_1,\ldots, C_s$ generate a strict digital
$(t_0,\alpha,\beta, n \times m,s)$-net over $\mathbb{F}_q$, where we
assume that $\beta n$ is an integer, then
\begin{equation*}
\delta_{\alpha,n}(\N) = \beta n - t_0 + 1.
\end{equation*}
\end{proposition}

\subsection{Duality theory for general higher order nets}

Here we review duality theory for higher order nets which also
applies to point sets not obtained by the digital construction
scheme. The basic tool are Walsh functions in integer base $b \ge
2$ whose definition and basic properties are recalled in the
following. We repeat some background from \cite{BDP11}, where also the proofs of the results mentioned in this section can be found.

\begin{definition}\rm
Let $b \geq 2$ be an integer and represent $k \in \mathbb{N}_0$ in
base $b$, $k=\kappa_{a-1} b^{a-1} + \dots + \kappa_0$, with
$\kappa_i \in \left\{ 0,\dots,b-1\right\}$. Further let
$\omega_b={\rm e}^{2 \pi \icomp/b}$ be the $b$th root of unity.
Then the $k${\it th} $b$-{\it adic Walsh function} $\wal_k(x)
:[0,1) \rightarrow \left\{ 1,\omega_b,\dots,\omega^{b-1}_b
\right\}$ is given by
\begin{displaymath}
 \wal_k(x)=\omega_b^{\xi_1 \kappa_0 + \dots + \xi_a \kappa_{a-1}},
\end{displaymath}
for $x \in [0,1)$ with base $b$ representation $x=\xi_1 b^{-1} +
\xi_2 b^{-2} + \cdots $ (unique in the sense that infinitely many
of the $\xi_i$ are different from $b-1$).

For dimension $s \geq 2$, $\bsx=(x_1,\dots,x_s) \in [0,1)^s$, and
$\bsk =(k_1,\dots,k_s) \in \mathbb{N}^s_0$, we define
$\wal_{\bsk}:[0,1)^s \rightarrow
\left\{1,\omega_b,\dots,\omega^{b-1}_b \right\}$ by
\begin{displaymath}
\wal_{\bsk}(\bsx)=\prod^{s}_{j=1} \wal_{k_j}(x_j).
\end{displaymath}
\end{definition}

The following notation will be used throughout the paper: By
$\oplus$ we denote digitwise addition modulo $b$, i.e., for
$x,y \in [0,1)$ with base $b$ expansions $x=\sum^{\infty}_{l=1}
\xi_{l} b^{-l}$ and $y=\sum^{\infty}_{l=1} \eta_{l} b^{-l}$, we
define
\begin{displaymath}
 x \oplus y =\sum^{\infty}_{l=1} \zeta_{l} b^{-l},
\end{displaymath}
where $\zeta_{l} \in \left\{0,\dots,b-1 \right\}$ is given by
$\zeta_{l} \equiv \xi_{l} + \eta_{l} \pmod{b}$. Let $\ominus$
denote digitwise subtraction modulo $b$ (for short we use
$\ominus x:=0 \ominus x$). In the same fashion we also define 
digitwise addition and digitwise subtraction of nonnegative
integers based on the $b$-adic expansion. For vectors in $[0,1)^s$
or $\mathbb{N}^s_0$, the operations $\oplus$ and $\ominus$ are
carried out componentwise. Throughout the paper, we always use the
same base $b$ for the operations $\oplus$ and $\ominus$ as is used
for Walsh functions. Further, we call $x \in [0,1)$ a $b$-adic
rational if it can be written in a finite base $b$ expansion. The
following simple properties of Walsh functions will be used
several times.

For all $k,l \in \mathbb{N}_0$ and all $x,y \in [0,1)$, with the
restriction that if $x,y$ are not $b$-adic rationals, then $x
\oplus y$ is not allowed to be a $b$-adic rational, we have
$\wal_k(x) \cdot \wal_l(x) = \wal_{k \oplus l}(x)$ and $\wal_k(x)
\cdot \wal_k(y)=\wal_k(x \oplus y)$. Furthermore,
$\overline{\wal_k(x)}=\wal_{\ominus k}(x)$.

Now we turn to duality theory for nets. Let
$\mathcal{K}^s_{b,r}=\{0,\ldots,b^r-1\}^s$. We also assume there
is an ordering of the elements in $\mathcal{K}^s_{b,r}$ which can
be arbitrary but needs to be the same in each instance, i.e., let
$\mathcal{K}^s_{b,r} = \{\bsk_0,\ldots, \bsk_{b^{sr}-1}\}$. (Note
that $|\mathcal{K}^s_{b,r}|=b^{s r}$.) By this we mean that in
expressions like $\sum_{\bsk \in \mathcal{K}^s_{b,r}}$,
$(a_{\bsk,\bsl})_{\bsk,\bsl \in \mathcal{K}^s_{b,r}}$, and
$(c_{\bsk})_{\bsk \in \mathcal{K}^s_{b,r}}$ the elements $\bsk$
and $\bsl$ run through the set $\mathcal{K}^s_{b,r}$ always in the
same order.

The following $b^{sr} \times b^{sr}$ matrix plays a central role
in the duality theory for higher order nets,
$${\bf
W}_r:=  \left(\wal_{\bsk}\left(b^{-r}\bsl\right)\right)_{\bsk,\bsl
\in \mathcal{K}^s_{b,r}}.$$ We call ${\bf W}_r$ a \emph{Walsh
matrix}.

In the following we denote by $A^\ast$ the conjugate transpose of
a matrix $A$ over the complex numbers $\mathbb{C}$, i.e., $A^\ast
= \overline{A}^\top$. The following lemma was shown in \cite{BDP11}.

\begin{lemma}\label{le1}
The Walsh matrix ${\bf W}_r$ is invertible and its inverse is
given by ${\bf W}_r^{-1}= b^{-sr} {\bf W}_r^{\ast}$.
\end{lemma}


Let now $b \ge 2$ and $r,N \ge 1$ be integers. For a multiset  $P=
\{\bsx_0,\ldots ,\bsx_{N -1}\}$ in $[0,1)^s$ and $\bsk \in
\mathcal{K}^s_{b,r}$ we define
$$c_{\bsk}=c_{\bsk}(P):=\sum_{h=0}^{N -1} \wal_{\bsk}(\bsx_h)$$ (note that
$|c_{\bsk}|\le N$ and $c_{\bszero} = N$) and the vector
\begin{equation} \label{eqdefvectorC} \vec{C}=\vec{C}(P):=(c_{\bsk})_{\bsk \in \mathcal{K}^s_{b,r}}.\end{equation}

For $\bsa=(a_1,\ldots,a_s) \in \mathcal{K}^s_{b,r}$ define the
elementary $b$-adic interval $$E_{\bsa}:=\prod_{j=1}^s
\left[\frac{a_{j}}{b^r},\frac{a_{j} +1}{b^r}\right).$$

We also have the following lemma.
\begin{lemma}\label{le2}
We have $$\sum_{\bsk \in \mathcal{K}^s_{b,r}} \wal_{\bsk}(\bsx
\ominus \bsy)=\left\{
\begin{array}{ll}
|\mathcal{K}^s_{b,r}| & \mbox{ if } \bsx,\bsy \in E_{\bsa} \mbox{ for some } \bsa \in \mathcal{K}^s_{b,r},\\
0 & \mbox{ otherwise}.
\end{array}\right.$$
\end{lemma}


Let $\bsx \in E_{\bsa}$ for some $\bsa \in \mathcal{K}^s_{b,r}$.
Then, using Lemma~\ref{le2}, we have
\begin{eqnarray*}
\frac{1}{|\mathcal{K}^s_{b,r}|}\sum_{\bsk \in \mathcal{K}^s_{b,r}}
c_{\bsk} \overline{\wal_{\bsk}(\bsx)} & =
& \frac{1}{|\mathcal{K}^s_{b,r}|}\sum_{\bsk \in \mathcal{K}^s_{b,r}} \sum_{h=0}^{N -1} \wal_{\bsk}(\bsx_h \ominus \bsx)\\
& = & \sum_{h=0}^{N -1} \frac{1}{|\mathcal{K}^s_{b,r}|}\sum_{\bsk \in \mathcal{K}^s_{b,r}}\wal_{\bsk}(\bsx_h \ominus \bsx)\\
& = & |\{h \, : \, \bsx_h \in E_{\bsa}\}|=:m_{\bsa}.
\end{eqnarray*}

\begin{definition}\rm\label{def_dual}
Let $b \ge 2$ and $r,N \ge 1$ be integers. Let $P =
\{\bsx_0,\ldots, \bsx_{N-1}\}$ be a multiset in $[0,1)^s$  and let
$\mathcal{K}^s_{b,r} = \{0,\ldots, b^r-1\}^s$.
\begin{enumerate}
\item For $\bsa \in \mathcal{K}^s_{b,r}$ let
$$m_{\bsa}=m_{\bsa}(P): = |\{h \, : \, \bsx_h \in E_{\bsa}\}|$$
and
$$ \vec{M}=\vec{M}(P): = \left(m_{\bsa}\right)_{\bsa \in \mathcal{K}^s_{b,r}}. $$
Then we call the vector $\vec{M}$ the \emph{point set vector}
(with resolution $r$). \item The vector $\vec{C}=\vec{C}(P)$ from
\eqref{eqdefvectorC} is called the \emph{dual vector} (with
respect to the Walsh matrix $\mathbf{W}_r$). \item The set
$$\mathcal{D}_r =\mathcal{D}_r(P):= \{\bsk \in
\mathcal{K}^s_{b,r}: c_{\bsk} \neq 0\}$$ is called the \emph{dual
set} (with respect to the Walsh matrix $\mathbf{W}_r$).
\end{enumerate}
\end{definition}

The relationship between a point set vector and its dual vector is
stated in the following theorem.

\begin{theorem}\label{thm_dual}
Let $P = \{\bsx_0,\ldots, \bsx_{N-1}\}$ be a multiset in $[0,1)^s$
and let $r \in \mathbb{N}$. Let $\vec{M}$ be the point set vector
with resolution $r$ and $\vec{C}$ be the dual vector with respect
to $\mathbf{W}_r$ defined as above. Then
\begin{equation}
\frac{1}{|\mathcal{K}^s_{b,r}|} {\bf W}_r \vec{C}=\vec{M}
\;\;\mbox{ and }\;\; \vec{C} = {\bf W}_r^{\ast} \vec{M}.
\end{equation}
\end{theorem}


The vector $\vec{C}$ carries sufficient information to construct a
point set in the following way: Given $\vec{C}$, we can use
Theorem~\ref{thm_dual} to determine how many points are to be
placed in the interval $E_{\bsa}$, $\bsa \in \mathcal{K}^s_{b,r}$.
We remark that at the functional level, the vector $\vec{C}$ is
comparable to the generator matrices of a digital net, which
completely determine the point set.

Note that for the $(t,\alpha,\beta,n,m,s)$-net property it is of
no importance where exactly within an interval $E_{\bsa}$, $\bsa
\in \mathcal{K}^s_{b,n}$, the points are placed. Hence we can
reconstruct a net from a dual vector with respect to
$\mathbf{W}_r$ provided that $r \ge \lfloor \beta n \rfloor - t$.
In words, if one knows the dual vector of a net, then one can use
this dual vector to obtain the net via Theorem~\ref{thm_dual}
provided that the resolution is greater than or equal to the
strength of the net.

In analogy, the dual space of a digital net also allows us to
reconstruct the original point set, see \cite{NP}. Although
$\vec{C}$ is different from the dual space for digital nets, it
contains the same information and can be used in a manner similar
to the dual space. In case $P$ is a digital $(t,\alpha,\beta,n
\times m,s)$-net, the dual set $\mathcal{D}_n$ defined in
Definition~\ref{def_dual} coincides with the dual space defined in
\cite{DK08} intersected with $\mathcal{K}^s_{b,n}$, and if $P$ is
a digital $(t,m,s)$-net, it coincides with the dual space in
\cite{NP} intersected with $\mathcal{K}^s_{b,m}$.

Although the above results hold for arbitrary point sets, in the
following we consider point sets which are nets and show how to
relate the quality of a $(t,\alpha,\beta,n,m,s)$-net to its dual
set. To this end we need to introduce a function which was first
introduced in \cite{D08} in the context of applying digital nets
to quasi-Monte Carlo integration of smooth functions and which is
related to the quality of suitable digital nets. For $k \in
\mathbb{N}_0$ and $\alpha \ge 1$ let
\begin{equation*}
\mu_\alpha(k) = \left\{\begin{array}{ll}  a_1 + \cdots +
a_{\min(\nu,\alpha)} & \mbox{for } k > 0, \\ 0 & \mbox{for } k =
0,
\end{array} \right.
\end{equation*}
where for $k > 0$ we assume that $k = \kappa_1 b^{a_1-1} + \cdots
+ \kappa_{\nu} b^{a_\nu-1}$ with $0 < \kappa_1,\dots, \kappa_\nu <
b$ and $1 \le a_\nu < \cdots < a_1$. Note that for $\alpha=1$ we
obtain the well-known Niederreiter-Rosenbloom-Tsfasman (NRT)
weight (see, for example, \cite[Section~7.1]{DPCam}).

For a vector $\boldsymbol{k} = (k_1,\dots, k_s) \in
\mathbb{N}_0^s$ we define $\mu_\alpha(\boldsymbol{k}) =
\mu_\alpha(k_1) + \cdots + \mu_\alpha(k_s)$ and for a subset
$\mathcal{Q}$ of $\mathcal{K}_{b,r}^s$ with $\mathcal{Q} \setminus
\{\bszero\} \neq \emptyset$ and $\alpha \ge 1$ define
$$\rho_\alpha(\mathcal{Q}):=\min_{\bsk \in \mathcal{Q} \setminus
\{\bszero\}} \mu_\alpha(\bsk).$$  For $\mathcal{Q} \subseteq
\{\bszero\}$ we set $\rho_\alpha(\mathcal{Q}) = r+1$.

Let $P= \{\bsx_0,\ldots, \bsx_{N-1}\} \subset [0,1)^s$. In the
following we consider for which cases we have $\mathcal{D}_r(P) =
\{\bszero\}$ (note that $\bszero \in \mathcal{D}_r(P)$ for any
point set $P$ with at least one point). If $\mathcal{D}_r(P) =
\{\bszero\}$, then we have $c_{\bszero} \neq 0$ and $c_{\bsk} = 0$
for all $\bsk \in \mathcal{K}_{b,r}^s \setminus\{\bszero\}$. By
Theorem~\ref{thm_dual} we have $\vec{M}(P) = c_{\bszero} b^{-rs}
(1, 1, \ldots, 1)^\top$, that is, each box $E_{\bsa}$ contains
exactly $c_{\bszero} b^{-rs}$ points for all $\bsa \in
\mathcal{K}_{b,r}^s$ and $P$ consists of $N = c_{\bszero}$ points
altogether. This is the only case for which $\mathcal{D}_r(P) =
\{\bszero\}$.

Conversely, since the number of points in $E_{\bsa}$ must be an
integer, it follows that $c_{\bszero} b^{-rs} \in \mathbb{N}$,
i.e., $b^{rs}$ divides $c_{\bszero}$ and therefore $b^{rs}$
divides $N$. From this we conclude that if we choose a resolution
$r \in \mathbb{N}$ such that $b^{rs} > N$, i.e., $r > \frac{1}{s}
\log_b N$, then $\mathcal{D}_r(P) \neq \{\bszero\}$. For a net
with $N = b^m$ points this means that we require $r > m/s$.

The following theorem establishes a relationship between
$\rho_{\alpha}(\mathcal{Q})$ and the quality of a
$(t,\alpha,\beta,n,m,s)$-net.

\begin{theorem} \label{theorhoalpha}
Let $P= \{\bsx_0,\ldots,\bsx_{b^m -1}\} \subset [0,1)^s$ be a
multiset. Then $P$ is a $(t,\alpha,\beta,n,m,s)$-net in base $b$
if and only if $\rho_\alpha(\mathcal{D}_{\lfloor \beta n \rfloor -
t})\ge \lfloor \beta n \rfloor -t+1$. If $P$ is a strict
$(t_0,\alpha,\beta, n, m, s)$-net in base $b$, then
$\rho_\alpha(\mathcal{D}_{\lfloor \beta n \rfloor - t_0}) =
\lfloor \beta n \rfloor - t_0 + 1$.
\end{theorem}

%

Let now $P$ be a strict $(t_0,\alpha,\beta, n,m,s)$-net in base
$b$. Let $r \ge \lfloor \beta n \rfloor - t_0$. Then
$\mathcal{D}_r \supseteq \mathcal{D}_{\lfloor \beta n \rfloor -
t_0}$ and $\mathcal{D}_r \setminus \mathcal{D}_{\lfloor \beta n
\rfloor - t_0} \subseteq \mathcal{K}_{b,r}^s \setminus
\mathcal{K}_{b,\lfloor \beta n \rfloor - t_0}^s$. For any $\bsk
\in \mathcal{K}_{b,r}^s \setminus \mathcal{K}_{b,\lfloor \beta n
\rfloor - t_0}^s$ we have $\mu_\alpha(\bsk) \ge \lfloor \beta n
\rfloor - t_0 + 1$. Theorem~\ref{theorhoalpha} implies that
$\rho_\alpha(\mathcal{D}_{\lfloor \beta n \rfloor - t_0}) =
\lfloor \beta n \rfloor - t_0 + 1$ and hence
$\rho_\alpha(\mathcal{D}_r) =  \rho_\alpha(\mathcal{D}_{\lfloor
\beta n \rfloor - t_0}) = \lfloor \beta n \rfloor - t_0 + 1$. In
particular, for all $r, r' \ge \lfloor \beta n \rfloor - t_0$ we
have
\begin{equation} \label{eqequaldualvec} \rho_\alpha(\mathcal{D}_r) = \rho_{\alpha}(\mathcal{D}_{r'}) =
\rho_\alpha(\mathcal{D}_n) = \lfloor \beta n \rfloor -
t_0+1,\end{equation} since $n \ge \lfloor \beta n \rfloor - t_0$.


\section{A generalized Blokh-Zyablov construction}\label{secbzconstr}

In this section we generalise the Blokh-Zyablov construction for codes from \cite{BZ74} and for classical nets from \cite{SS09}. The proofs closely follow those in \cite{SS09}.

We call $\mathcal{N} \subseteq \mathbb{F}_q^{ns}$ an
$((s,n),\alpha,K,\delta)_q$-space if $K = |\mathcal{N}|$ and
$\delta = \delta_{\alpha,n}(\mathcal{N})$. If $\mathcal{N}$ is a
linear subspace of $\mathbb{F}_q^{ns}$ then we call $\mathcal{N}$
a linear $[s,n,\alpha,k,\delta]_q$-space if $|\mathcal{N}| = q^k$
and $\delta = \delta_{\alpha,n}(\mathcal{N})$.

Note that an $((s,n),1,K,\delta)_q$-space is a generalized
$((s,n),K,\delta)_q$-code and a linear
$[(s,n),1,k,\delta]_q$-space is a generalized linear
$[(s,n),k,\delta]_q$-code in the sense of Sch\"urer and
Schmid~\cite{SS09, mint}. Note, furthermore, that in the special
case $n=1$ an $[(s,1),k,\delta]_q$-code is an ordinary linear
$[s,k,\delta]_q$-code (see \cite{SS09} or \cite{mint} for
details).

\begin{lemma}\label{lem1b}
Let $\alpha, \alpha' \ge 1$ and let
\begin{equation*}
\{\mathbf{0}\} = \mathcal{N}^\prime_0 \subset \mathcal{N}^\prime_1
\subset \cdots \subset \mathcal{N}^\prime_r = \mathbb{F}_q^{n's'}
\end{equation*}
be a chain of linear $[s',n',\alpha',k'_u,
\delta'_{u}]_q$-spaces and let
$\mathbf{v}_1,\ldots, \mathbf{v}_{n's'}\in \mathbb{F}_q^{n's'}$
denote vectors such that $\mathcal{N}^\prime_u$ is generated by
$\mathbf{v}_1,\ldots, \mathbf{v}_{k'_u}$ for $u = 1,\ldots, r$. We call the spaces $\mathcal{N}^\prime_u$ the inner spaces.

Let $\mathcal{N}_u$ denote (not necessarily linear)
$((s,n),\alpha, K_u,\delta_u)_{q_u}$-spaces with $q_u = q^{e_u}$
and $e_u = k'_{u} - k'_{u-1}$ for $u = 1,\ldots, r$. We call the spaces $\mathcal{N}_u$  the outer
spaces.

Let $\mathcal{N}$ be given by
\begin{equation}\label{def_N}
\mathcal{N} = \left\{\sum_{u=1}^r \phi_u(\mathbf{x}_u):
\mathbf{x}_u \in \mathcal{N}_u \mbox{ for } u  = 1,\ldots, r
\right\},
\end{equation}
where $\phi_u:\mathbb{F}_{q_u}^{ns} \to \mathbb{F}_q^{nn'ss'}$
replaces each symbol (regarded as a vector of length $e_u$ over
$\mathbb{F}_q$) of an element from $\mathcal{N}_u$ by the
corresponding linear combination of the $e_u$ vectors
$\mathbf{v}_{k'_{u-1}+1}, \ldots, \mathbf{v}_{k'_u}$. The elements
in the resulting vector are grouped such that column $(a', \tau')$
(with $1 \le a' \le s'$ and $1 \le \tau' \le n'$) from
$\mathcal{N}'_u$ and column $(a, \tau)$ (with $1 \le a \le s$ and
$1 \le \tau \le n$) from $\mathcal{N}_u$ determine column $((a-1)
s' + a', (\tau-1) n' + \tau')$ in the resulting element in
$\mathcal{N}$.

Then $\mathcal{N}$ is an
\begin{equation*}
((ss',nn'),\alpha \alpha', K_1\cdots K_r, \min_{1 \le u \le r,
|\mathcal{N}_u| > 1} \delta_u \delta'_u)_q-\mathrm{space}.
\end{equation*}

\end{lemma}

\begin{proof}
The proof follows along the same lines as the proof of
\cite[Theorem~3]{mint}. We have $\mathcal{N} \subseteq
\mathbb{F}_q^{nn'ss'}$ by the definition of the mappings $\phi_u$.
The number of elements in $\mathcal{N}$ is given by $K_1\cdots
K_r$ which follows from \eqref{def_N} and the fact that all
elements given by $\sum_{u=1}^r \phi_u(\mathbf{x}_u)$ for
$\mathbf{x}_u \in \mathcal{N}_u$ and $u  = 1,\ldots, r$ are
distinct, which we show in the following.

Let $\mathbf{x}, \mathbf{y}$ be two elements
$\mathcal{N}$ and let $\mathbf{x}_u, \mathbf{y}_u \in
\mathcal{N}_u$, $1\le u\le r$, denote the elements defining
$\mathbf{x}$ and $\mathbf{y}$ using the sum in \eqref{def_N}. If
$\mathbf{x}_u = \mathbf{y}_u$ for $1 \le u \le r$, then it follows
that $\mathbf{x} = \mathbf{y}$. Assume now that $\mathbf{x} \neq
\mathbf{y}$, then there exists a largest integer $1 \le u^\ast \le
r$ such that $\mathbf{x}_{u^\ast} \neq \mathbf{y}_{u^\ast}$. This
implies that $|\mathcal{N}_{u^\ast}| > 1$.

Since $\phi_u$ is linear, we have
\begin{equation*}
\mathbf{x}-\mathbf{y} = \sum_{u=1}^{u^\ast-1} \phi_u(\mathbf{x}_u
- \mathbf{y}_u) + \phi_{u^\ast}(\mathbf{x}_{u^\ast} -
\mathbf{y}_{u^\ast}) +\sum_{u=u^\ast+1}^r
\phi_u(\mathbf{x}_u-\mathbf{y}_u),
\end{equation*}
where the first sum is a concatenation of elements from
$\mathcal{N}'_{u^\ast-1}$ and the second sum is $\mathbf{0}$
(since $\mathbf{x}_u = \mathbf{y}_u$ and the linearity of $\phi_u$
for $u^\ast < u \le r$). Further,
$\phi_{u^\ast}(\mathbf{x}_{u^\ast} -\mathbf{y}_{u^\ast})$ is a
concatenation of elements from $\mathcal{N}'_{u^\ast} \setminus
\mathcal{N}'_{u^\ast-1}$. We denote the elements in this
concatenation by $\mathbf{w}^{j,k}$ for $1\le j\le s$ and $1\le
k\le n$.

Thus, $\mathbf{x}-\mathbf{y}$ is a concatenation of elements from
$\mathcal{N}'_{u^\ast} \setminus \mathcal{N}'_{u^\ast-1}$, since
the sum of an element from $\mathcal{N}'_{u^\ast-1}$ and an
element from $\mathcal{N}'_{u^\ast} \setminus
\mathcal{N}'_{u^\ast-1}$ is again an element of
$\mathcal{N}'_{u^\ast} \setminus \mathcal{N}'_{u^\ast-1}$.

Let
\begin{equation*}
\mu_{\alpha,n}(\mathbf{x}_{u^\ast} - \mathbf{y}_{u^\ast}) =
\sum_{j=1}^s \sum_{k=1}^{\min(\alpha, \nu_j)} i_{j,k},
\end{equation*}
where $\nu_j$ denotes the number of nonzero elements and $i_{j,k}$
are the positions of the nonzero elements in
$\mathbf{x}_{u^\ast}-\mathbf{y}_{u^\ast}$. Furthermore, let
\begin{equation*}
\mu_{\alpha',n'}(\mathbf{w}^{(j,k)}) = \sum_{j'=1}^{s'}
\sum_{k'=1}^{\min(\alpha', \nu'_{j',j,k})} i'_{j',k',j,k},
\end{equation*}
where $\nu'_{j',j,k}$ denotes the number of nonzero elements and
$i'_{j',k',j,k}$ denotes the positions of the nonzero elements in
$\mathbf{w}^{(j,k)}$. Then we have
\begin{align*}
d_{\alpha \alpha',n n'}(\mathbf{x},\mathbf{y}) & =
\mu_{\alpha\alpha', nn'}(\mathbf{x}-\mathbf{y}) \\ & \ge
\sum_{j=1}^s \sum_{j'=1}^{s'}\sum_{k=1}^{\min(\alpha,\nu_j)}
\sum_{k'=1}^{\min{(\alpha',\nu'_{j',j,k})}}
\left\{\begin{array}{ll} (i_{j,k}-1) n' +  i'_{j',k',j,k}  &
\mbox{if } \nu_j, \nu'_{j'} > 0
\\ 0 & \mbox{otherwise} \end{array}\right. \\ & \ge \sum_{j=1}^s
\sum_{j'=1}^{s'}\sum_{k=1}^{\min(\alpha,\nu_j)}
\sum_{k'=1}^{\min{(\alpha',\nu'_{j',j,k})}} i_{j,k}
i'_{j',k',j,k}\\ & = \sum_{j=1}^s  \sum_{k=1}^{\min(\alpha,\nu_j)}
i_{j,k} \sum_{j'=1}^{s'}
\sum_{k'=1}^{\min{(\alpha',\nu'_{j',j,k})}} i'_{j',k',j,k} \\ &
\ge \mu_{\alpha,n}(\mathbf{x}_{u^\ast}-\mathbf{y}_{u^\ast})
\mu_{\alpha',n'}(\mathbf{w}^{(j,k)}) \\ & \ge
\delta_{\alpha,n}(\mathcal{N}_{u^\ast})
\delta_{\alpha',n'}(\mathcal{N}'_{u^\ast}).
\end{align*}
The result thus follows.
\end{proof}

\begin{lemma}\label{lem2}
Let $\{\mathbf{0}\} = \mathcal{N}'_0 \subset \mathcal{N}_1'
\subset \cdots \subset \mathcal{N}'_r = \mathbb{F}_q^{s'}$ be as
in Lemma~\ref{lem1b} with $\alpha' = n'=1$. Let
$\mathbf{v}_1,\ldots, \mathbf{v}_{s'} \in \mathbb{F}_q^{s'}$ be
defined as in Lemma~\ref{lem1b}, where we assume without loss of
generality that the first $i-1$ elements of $\mathbf{v}_i$ are
$0$. This means that $(\mathbf{v}_1,\ldots, \mathbf{v}_{s'})^\top$
forms an $s' \times s'$ upper triangular matrix.

Moreover, let $q_u$ and $e_u$ be defined as in Lemma~\ref{lem1b}. For
positive integers $s_1 \le s_2 \le \cdots \le s_r$ let
$\mathcal{N}_u$, $u = 1,\ldots, r$, be a (not necessarily linear) $((s_u, n),
\alpha, K_u, \delta_u)_{q_u}$-space.

Then we can construct an
\begin{equation*}
((s_1 e_1+ \cdots + s_r e_r, n),\alpha, K_1\cdots K_r, \min_{1 \le
u \le r} \delta_u \delta'_u)_q-\mathrm{space}.
\end{equation*}
\end{lemma}

\begin{proof}
The construction is analogous to the construction in the proof of
\cite[Theorem~4]{mint} and follows in three steps.
\begin{itemize}
\item Let $s = s_r$ and construct new spaces $\mathcal{M}_u$ by
embedding each space $\mathcal{N}_u$ in $\mathbb{F}_{q_u}^{(s,n)}$. This is achieved
by prepending $\mathbf{0} \in \mathbb{F}_{q_u}^{(s-s_u,n)}$ to
each element of $\mathcal{N}_u$. 

\item Use Lemma~\ref{lem1b} with $\mathcal{N}'_0,\ldots, \mathcal{N}'_r$ as
inner spaces and $\mathcal{M}_1,\ldots, \mathcal{M}_r$ as outer
spaces. This yields a new space $\mathcal{M}$.

\item Finally, for $i = 1, \ldots, s'$, choose $u$ minimal such that the condition
$e_1 +\cdots + e_u \ge i$ holds and construct a space $\mathcal{N}$ from
$\mathcal{M}$. This can be done by deleting all columns $(k,\tau)$ in $\mathcal{M}$
with $k = 0 s' + i, 1 s'+i, \ldots, (s-s_u-1) s' + i$.

The total number of deleted blocks is $e_1 (s-s_1) + \cdots + e_r
(s-s_r)$, so the length of $\mathcal{N}$ is $s_1 e_1 + \cdots +
s_r e_r$. The deleted positions are $\mathbf{0} \in
\mathbb{F}_q^{n}$ for each $\phi_u(\mathbf{x}_u)$, either due to a
$0$ in $\mathbf{v}_i$ or due to a $\mathbf{0}$ appended to
$\mathbf{x}_u \in \mathcal{N}_u$. Note that this procedure neither influences the dimension nor
the weight of $\mathcal{M}$.
\end{itemize}
\end{proof}

\section{A Blokh-Zyablov propagation rule}\label{sec-blokh}

\subsection{A Blokh-Zyablov propagation rule for digital higher order nets}

Lemma~\ref{lem2} can be applied to digital higher order nets which
yields a new propagation rule. This propagation rule generalises
the Matrix-product construction in \cite{DK08}. As there were
Propagation Rules I--XIV in \cite{DK08}, we call our new rule
Propagation Rule XV.

\begin{theorem}[Propagation Rule XV]\label{thmbz}
Let $\mathbf{v}_1,\ldots, \mathbf{v}_{s'}$ and $\mathcal{N}'_u$,
$q_u$ and $e_u$ for $u=1,\ldots, r$ be defined as in
Lemma~\ref{lem2}. Let $\alpha \ge 1$. For positive integers $s_1
\le s_2 \le \cdots \le s_r$ let $\mathcal{P}_u$, $1\le u\le r$,
denote digital $(t_u, \alpha,\beta_u, n_u \times m_u, s_u)$-nets
over $\mathbb{F}_{q_u}$.

Then a digital $(t, \alpha, \beta, n \times m, s)$-net over
$\mathbb{F}_q$ can be constructed, where
\begin{align*}
s & = \sum_{u=1}^r e_u s_u, \\
m & = \sum_{u=1}^r e_u m_u, \\
n & = \sum_{u=1}^r e_u n_u, \\
\beta & = \min(1, \alpha m/n), \\
t & \le \beta n + 1 - \min_{1 \le u \le r} (\beta_u n_u-t_u+1)
\delta'_u.
\end{align*}
\end{theorem}

\begin{proof}
Let $\mathcal{N}_u$ denote the dual space of $\mathcal{P}_u$ for
$u=1,\ldots, r$. Then $\mathcal{N}_u$ is a linear $[(s,n),\alpha,
m_u, \beta_u n_u - t_u + 1]_{q_u}$-space. The result then follows
by applying Lemma~\ref{lem2} using $\mathcal{N}'_0,\ldots,
\mathcal{N}'_r$ as inner spaces and $\mathcal{N}_1,\ldots,
\mathcal{N}_r$ as outer spaces to obtain a space $\mathcal{N}$ and
applying Proposition~\ref{thm1}.
\end{proof}

\subsection{A Blokh-Zyablov propagation rule for higher order nets}
In this subsection we consider higher order nets which are not
necessarily digital. Again we have the same theorem as in the
previous case by considering the dual set instead of the dual net.
Applying Lemma~\ref{lem2} and Theorem~\ref{theorhoalpha} yields
the following result, which is Propagation Rule 17 for general
higher order nets (as there were Propagation Rules 1--16 for
general higher order nets in \cite{BDP11}).

\begin{theorem}[Propagation Rule 17]\label{thmgenbz}
Let $\mathbf{v}_1,\ldots, \mathbf{v}_{s'}$ and $\mathcal{N}'_u$,
$q_u$ and $e_u$ for $u=1,\ldots, r$ be defined as in
Lemma~\ref{lem2}. Let $\alpha \ge 1$. For positive integers $s_1
\le s_2 \le \cdots \le s_r$ let $\mathcal{P}_u$ denote $(t_u,
\beta_u, \alpha, n_u \times m_u, s_u)$-nets in base $b_u =
b^{e_u}$ for $1 \le u \le r$.

Then a $(t, \alpha, \beta, n \times m, s)$-net in base $b$ can be
constructed, where
\begin{align*}
s & = \sum_{u=1}^r e_u s_u, \\
m & = \sum_{u=1}^r e_u m_u, \\
n & = \sum_{u=1}^r e_u n_u, \\
\beta & = \min(1, \alpha m/n), \\
t & \le \beta n + 1 - \min_{1 \le u \le r} (\beta_u n_u-t_u+1)
\delta'_u.
\end{align*}
\end{theorem}
\begin{remark}\label{remsubcode}
Note that in Lemma \ref{lem1b}, Lemma \ref{lem2}, Theorem \ref{thmbz} and Theorem
\ref{thmgenbz} it is sufficient to require $\mathcal{N}'_r
\subseteq \mathbb{F}_q^{s'}$ instead of $\mathcal{N}'_r =
\mathbb{F}_q^{s'}$. Indeed, this slight modification is possible as one can use a trivial space 
$\mathcal{N}_u=\{\bszero\}$ as an outer space in Lemma \ref{lem1b}, which does not influence the minimal distance of the newly obtained space (cf. Remark 4 in \cite{SS09}).
\end{remark}

\section{Numerical results}\label{secnum}

Theorems \ref{thmbz} and \ref{thmgenbz} provide previously unknown
propagation rules for higher order nets. To show how powerful these are, 
we state exemplary numerical results for the
digital case, but it should be noted that the results are also true for the more general case,
as the propagation rules yield the same parameters 
and we always start from existing digital nets, which are by definition
a subclass of general higher order nets. In Tables \ref{table1}--\ref{table3}, we present, for selected
values of $m$ and $s$, the results obtained when we use different
propagation rules for the case $q=5$, $\alpha=2$, and $\beta=1$.
As we restrict ourselves to considering
only digital nets, we only refer to propagation rules
with roman numbers within this section.

Table \ref{table1} covers the case where $s=5$,
Table \ref{table2} the case $s=15$ and Table
\ref{table3} the case $s=25$. In all tables we consider $m$
between 15 and 30. Since $\alpha$ and $\beta$ are fixed, and
different propagation rules might yield different ratios of $n$
and $m$, it is most useful to compare the strengths of the nets
obtained. As outlined above, the strength of a digital net refers
to the value of $\sigma=\sigma(\beta,n,t)=\beta n -t$. Note that,
in some of our new propagation rules, one can make many different
choices of smaller nets that might yield a bigger
net with the same parameters. We only give the best
values of the strength $\sigma$ we can obtain by going through a number of
possible choices of the smaller nets involved.

Our tables below can be seen as an extension of Table 3 in
\cite[Section 4]{DK08}. To be more precise, 
we consider the following quantities in Tables
\ref{table1} and \ref{table2}.
\begin{itemize}
 \item $\sigma_{\mathrm{dir}}$: The strength of a digital $(t,2,1,2m \times m,s)$-net over $\Field_5$ using
 the generating matrices of an existing classical digital $(t',m,2s)$-net over $\Field_5$, where we then obtain (cf.~\cite{D08})
\begin{equation}\label{eqtbounddirect}
t\le 2\min\left\{m,t' + \left\lfloor\frac{s}{2}\right\rfloor\right\}.
\end{equation}
We call this construction method the direct construction method.

\item $\sigma_{\mathrm{VII}}$: The strength of a digital net
constructed from a digital $(t_1,2,1,2 m_1\times m_1,s_1)$-net
$P_1$ and a digital $(t_2,2,1,2m_2\times m_2,s_2)$-net $P_2$ over
$\Field_5$ using Propagation Rule VII in \cite{DK08}, where $P_1$
and $P_2$ are obtained by the direct construction method from
classical nets. Here, $n=2m$.

\item $\sigma_{\mathrm{VIII}}$: The strength of a digital net
constructed from a digital $(t_1,2,1,2m_1\times m_1,s_1)$-net
$P_1$ and a digital $(t_2,2,1,2m_2\times m_2,s_2)$-net $P_2$
($s_1\le s_2$) over $\Field_5$ using Propagation Rule VIII in
\cite{DK08}, where $P_1$ and $P_2$ are obtained by the direct
construction method. Here, $n=2m$.

\item $\sigma_{\mathrm{IX}}$: The strength of a digital net
constructed from a digital $(t_1,2,1,2m_1\times m_1,s_1)$-net
$P_1$, a digital $(t_2,2,1,2m_2\times m_2,s_2)$-net $P_2$, and a
digital $(t_3,2,1,2m_3\times m_3,s_3)$-net $P_3$ ($s_1\le s_2\le
s_3$) over $\Field_5$ using Propagation Rule IX in \cite{DK08},
where $P_1$, $P_2$, and $P_3$ are obtained by the direct
construction method. Again, $n=2m$.

\item $\sigma_{\mathrm{XI}}$: The strength of a digital net
obtained by using Propagation Rule XI from \cite{DK08}  with
$r=2$, where the higher order nets plugged into Rule XI are
obtained by the direct construction method from classical nets.
Note that this rule can, since $r=2$, only be applied for the
cases where $m$ is even, and we then have $n=m$. As Propagation Rule XI with $r=2$
only yields higher order nets with an even value of $s$, we obtain
the values in the Tables \ref{table1}--\ref{table3} by
projection from 6-dimensional, 16-dimensional, and 26-dimensional nets,
respectively (this is allowed due to Propagation Rule V in
\cite{DK08}).

\item $\sigma_{\mathrm{XI}+\mathrm{VIII}}$: The strength of a
digital net obtained by first applying the direct construction
method to classical nets, then applying Propagation Rule XI with
$r=2$ in \cite{DK08}, and then using pairs of the newly obtained
nets to apply Propagation Rule VIII in \cite{DK08}. Note again
that we need to restrict ourselves to even cases of $m$ here, as
we first apply Rule XI with $r=2$. For
$\sigma_{\mathrm{XI}+\mathrm{VIII}}$, we again have $n=m$.

\item $\sigma_{\mathrm{XV}}$: The strength of a digital net by
first applying the direct construction method to classical nets,
and then Propagation Rule XV with $r=2$. (Theorem \ref{thmbz}). As
the inner spaces (see the proof of Theorem \ref{thmbz}, and Lemma
\ref{lem2}) we use (extended) Reed Solomon Codes over
$\mathbb{F}_5$. These codes are linear $[5,j,6-j]$-codes over
$\mathbb{F}_5$ (see \cite{mint}). In view of Remark
\ref{remsubcode}, we can use Reed Solomon Codes with parameters
$[5,0,6]_5$, $[5,2,4]_5$, and $[5,4,2]_5$ as the inner codes, hence $r$
indeed equals 2 and $e_1=e_2=2$ in Theorem \ref{thmbz}, which is
why we again need to restrict ourselves to even cases of $m$ here (in general one can of course also obtain odd values for $m$ by different choices of $e_i$ and $m_i$).
In the language of Theorem \ref{thmbz}, we have $\delta_1=4$ and
$\delta_2=2$ in this case. For $\sigma_{\mathrm{XV}}$ we have
$n=2m$, and, as for $\sigma_{\mathrm{XI}}$, we obtain the values
in the Tables \ref{table1}--\ref{table3} by projection from 6-,
16-, and 26-dimensional nets, respectively (again by
Propagation Rule V in \cite{DK08}).
\end{itemize}
The results in Tables \ref{table1}--\ref{table3} show that our new Propagation Rule XV is superior to the other 
propagation rules tested for dimensions 15 and 25, and that also the combination of other rules, such as that of Rules XI and VIII, does not yield better results.
In dimension 5, though still competitive, Propagation Rule XV is outperformed by the direct construction method. It is likely that the direct construction performs better
for lower dimensions as the $t$-values of classical digital nets are very small in this case. Furthermore, the bound \eqref{eqtbounddirect} depends critically on the value of $s$, so
it is natural that lower dimensional results for the direct construction method are stronger. The fact that the direct construction method works well for low dimensions
is in line with numerical results in \cite{DKPS07}, where a similar phenomenon was observed.

\begin{table}
\begin{center}

{\footnotesize
\begin{tabular}[b]{|c||c|c|c|c|c|c|c|}\hline
               $m$ &$\sigma_{\mathrm{dir}}$&$\sigma_{\mathrm{VII}}$& $\sigma_{\mathrm{VIII}}$&$\sigma_{\mathrm{IX}}$
               &$\sigma_{\mathrm{XI}}$&$\sigma_{\mathrm{XI}+\mathrm{VIII}}$&$\sigma_{\mathrm{XV}}$ \\
                \hline\hline
              $15$ &  24 &  12 &  17 &  14 &   &   &   \\ \hline
              $16$ &  26 &  14 &  18 &  16 & 14&  9& 19\\ \hline
              $17$ &  28 &  14 &  20 &  17 &   &   &   \\ \hline
              $18$ &  30 &  16 &  21 &  18 & 16& 10& 21\\ \hline
              $19$ &  32 &  16 &  22 &  20 &   &   &   \\ \hline
              $20$ &  34 &  18 &  24 &  20 & 18& 12& 25\\ \hline
              $21$ &  36 &  18 &  25 &  21 &   &   &   \\ \hline
              $22$ &  38 &  20 &  26 &  22 & 20& 13& 27\\ \hline
              $23$ &  40 &  20 &  28 &  24 &   &   &   \\ \hline
              $24$ &  42 &  22 &  29 &  25 & 22& 14& 29\\ \hline
              $25$ &  44 &  22 &  30 &  26 &   &   &   \\ \hline
              $26$ &  46 &  24 &  32 &  26 & 24& 16& 33\\ \hline
              $27$ &  48 &  24 &  33 &  28 &   &   &   \\ \hline
              $28$ &  50 &  26 &  34 &  29 & 26& 17& 35\\ \hline
              $29$ &  52 &  26 &  36 &  30 &   &   &   \\ \hline
              $30$ &  54 &  28 &  37 &  32 & 28& 18& 37\\ \hline
\end{tabular}}
\end{center}
\caption{$\sigma$-values depending on $m$ ($15\le m\le 30$) for
$\alpha=2$, $\beta=1$, $q=5$, and $s=5$.}\label{table1}
\end{table}

\begin{table}
\begin{center}

{\footnotesize
\begin{tabular}[b]{|c||c|c|c|c|c|c|c|}\hline
               $m$ &$\sigma_{\mathrm{dir}}$&$\sigma_{\mathrm{VII}}$& $\sigma_{\mathrm{VIII}}$&$\sigma_{\mathrm{IX}}$&$\sigma_{\mathrm{XI}}$
               &$\sigma_{\mathrm{XI}+\mathrm{VIII}}$&$\sigma_{\mathrm{XV}}$\\ \hline\hline
               $15$ &  0 &  4 &  5 &  6 &  &   &   \\ \hline
              $16$ &  2 &  4 &  5 &  8 &  8&  6& 13\\ \hline
              $17$ &  2 &  4 &  6 &  8 &   &   &   \\ \hline
              $18$ &  4 &  6 &  8 &  9 & 10&  8& 17\\ \hline
              $19$ &  4 &  6 &  9 & 10 &   &   &   \\ \hline
              $20$ &  6 &  6 & 10 & 12 & 12&  9& 19\\ \hline
              $21$ &  8 &  8 & 12 & 13 &   &   &   \\ \hline
              $22$ &  8 &  8 & 13 & 14 & 14&  9& 21\\ \hline
              $23$ & 10 & 10 & 13 & 14 &   &   &   \\ \hline
              $24$ & 12 & 12 & 14 & 16 & 16& 12& 25\\ \hline
              $25$ & 14 & 14 & 16 & 17 &   &   &   \\ \hline
              $26$ & 16 & 16 & 17 & 18 & 18& 13& 27\\ \hline
              $27$ & 18 & 18 & 18 & 20 &   &   &   \\ \hline
              $28$ & 20 & 20 & 20 & 20 & 20& 14& 29\\ \hline
              $29$ & 22 & 22 & 22 & 22 &   &   &   \\ \hline
              $30$ & 24 & 24 & 24 & 24 & 22& 16& 33\\ \hline
\end{tabular}}
\end{center}
\caption{$\sigma$-values depending on $m$ ($15\le m\le 30$) for
$\alpha=2$, $\beta=1$, $q=5$ and $s=15$.}\label{table2}
\end{table}

\begin{table}
\begin{center}

{\footnotesize
\begin{tabular}[b]{|c||c|c|c|c|c|c|c|}\hline
               $m$ &$\sigma_{\mathrm{dir}}$&$\sigma_{\mathrm{VII}}$& $\sigma_{\mathrm{VIII}}$&$\sigma_{\mathrm{IX}}$
               &$\sigma_{\mathrm{XI}}$&$\sigma_{\mathrm{XI}+\mathrm{VIII}}$&$\sigma_{\mathrm{XV}}$ \\
                \hline\hline
              $15$ &  0 &  0 &  1 &  2 &   &   &   \\ \hline
              $16$ &  0 &  0 &  1 &  2 &  4&  2&  5\\ \hline
              $17$ &  0 &  0 &  1 &  2 &   &   &   \\ \hline
              $18$ &  0 &  0 &  1 &  2 &  6&  4&  9\\ \hline
              $19$ &  0 &  0 &  1 &  2 &   &   &   \\ \hline
              $20$ &  0 &  0 &  1 &  2 &  8&  5& 11\\ \hline
              $21$ &  0 &  0 &  1 &  4 &   &   &   \\ \hline
              $22$ &  0 &  0 &  1 &  4 & 10&  6& 13\\ \hline
              $23$ &  0 &  2 &  2 &  5 &   &   &   \\ \hline
              $24$ &  0 &  2 &  2 &  5 & 12&  8& 17\\ \hline
              $25$ &  0 &  2 &  4 &  6 &   &   &   \\ \hline
              $26$ &  0 &  2 &  4 &  6 & 14&  9& 19\\ \hline
              $27$ &  2 &  4 &  5 &  8 &   &   &   \\ \hline
              $28$ &  2 &  4 &  5 &  8 & 16& 10& 21\\ \hline
              $29$ &  4 &  4 &  5 &  9 &   &   &   \\ \hline
              $30$ &  6 &  6 &  6 & 10 & 18& 12& 25\\ \hline
\end{tabular}}
\end{center}
\caption{$\sigma$-values depending on $m$ ($15\le m\le 30$) for
$\alpha=2$, $\beta=1$, $q=5$, and $s=25$.}\label{table3}
\end{table}

We emphasise that our examples are just illustrations and can by
no means systematically cover all cases one might theoretically
consider; to be more precise, we have the following restrictions
in Tables \ref{table1}--\ref{table3}.
\begin{itemize}
 \item We only show particular choices of the parameters involved. We restrict ourselves to some illustrative cases.

\item We do not consider all possible combinations of different
propagation rules, as this would lead to a too high number of
parameters.

\item Not all propagation rules are applicable for all sets of
parameters. This is indicated by void cells in the tables in cases
where a certain propagation rule was not applicable.
\end{itemize}

\section{Acknowledgements}

The authors would like to thank G. Pirsic, R. Sch\"{u}rer, and A.
Winterhof for valuable comments. Furthermore, P. Kritzer would like to thank 
J. Dick, F.Y. Kuo and I.H. Sloan for their hospitality during his visits to the University of
New South Wales in February 2011 and February 2012, during which parts of this paper were written.

\begin{small}
\noindent\textbf{Authors' addresses:} \ 
\\ \\
\noindent Josef Dick\\
School of Mathematics and Statistics, University of New South Wales\\ 
Sydney, NSW, 2052, Australia\\
\texttt{josef.dick@unsw.edu.au}\\ \\
\noindent Peter Kritzer\\ 
Institut f\"{u}r Finanzmathematik, Universit\"{a}t Linz\\ 
Altenbergerstr.~69, 4040 Linz, Austria\\
\texttt{peter.kritzer@jku.at}
\end{small}


\begin{thebibliography}{99}

\bibitem{BDP09} J. Baldeaux, J. Dick, F. Pillichshammer.
A characterisation of higher order nets using Weyl sums and its
applications. Unif. Distrib. Theory 5, 133--155, 2010.

\bibitem{BDP11} J. Baldeaux, J. Dick, F. Pillichshammer.
Duality theory and propagation rules for higher order nets.
Discrete Math. 311, 362--386, 2011.

\bibitem{bieredelschmid} J. Bierbrauer, Y. Edel, W.Ch. Schmid.
Coding-theoretic constructions for $(t,m,s)$-nets and ordered
orthogonal arrays. J. Comb. Des. 10, 403--418, 2002.

\bibitem{BZ74} E.L. Blokh, V.V. Zyablov. Coding of generalized concatenated codes.
Problems of Information Transmission, 10, 218--222, 1974.

\bibitem{D07}
J. Dick. Explicit constructions of quasi-Monte Carlo rules for the
numerical integration of high-dimensional periodic functions. SIAM
J. Numer. Anal. 45, 2141--2176, 2007.


\bibitem{D08} J. Dick. Walsh spaces containing smooth functions and
quasi-Monte Carlo rules of arbitrary high order. SIAM J. Numer.
Anal. 46, 1519--1553, 2008.


\bibitem{D09} J. Dick. On quasi-Monte Carlo rules achieving higher order convergence.
In: P. L'Ecuyer, A.B. Owen (eds.), {\it Monte Carlo and
Quasi-Monte Carlo Methods 2008}, Springer, 2010, pp. 73--96.

\bibitem{DickBalGeomprop} J. Dick, J. Baldeaux. Equidistribution properties of generalized nets and sequences.
In: P. L'Ecuyer, A.B. Owen (eds.), {\it Monte Carlo and
Quasi-Monte Carlo Methods 2008}, Springer, 2010, pp. 305--322.

\bibitem{DK08} J. Dick, P. Kritzer.
Duality theory and propagation rules for generalized digital nets.
Math. Comp. 79, 993--1017, 2010.

\bibitem{DKPS07} J. Dick, P. Kritzer, F. Pillichshammer, W.Ch. Schmid. On the existence of higher order polynomial lattices based on a generalized
figure of merit. J. Complexity 23, 581--593, 2007.

\bibitem{DPCam} J. Dick, F. Pillichshammer. {\it Digital Nets and Sequences. Discrepancy Theory and Quasi-Monte Carlo Integration}.
Cambridge University Press, Cambridge, 2010.

\bibitem{N87} H. Niederreiter. Point sets and sequences with small discrepancy. Monatsh. Math. 104, 273--337, 1987.

\bibitem{niesiam} H. Niederreiter. \textit{Random Number Generation and Quasi-Monte Carlo Methods}.
 CBMS--NSF Series in Applied Mathematics 63. SIAM, Philadelphia, 1992.

\bibitem{NiedNets} H. Niederreiter. Nets, $(t,s)$-sequences, and codes.
In: A. Keller, S. Heinrich, H. Niederreiter (eds.), {\it Monte
Carlo and Quasi-Monte Carlo Methods 2006}, Springer, 2008,
pp.83--100.

\bibitem{NO} H. Niederreiter, F. \"Ozbudak. Matrix-product constructions of digital nets. Finite Fields Appl. 10, 464--479, 2004.

\bibitem{NP} H. Niederreiter, G. Pirsic. Duality for digital nets and its applications.  Acta Arith. 97, 173--182, 2001.

\bibitem{niexing98} H. Niederreiter, C.P. Xing.
Nets, $(t,s)$-sequences, and algebraic geometry. In: P.
Hellekalek, G. Larcher (eds.), \textit{Random and Quasi-Random
Point Sets}. Springer, New York, 1998, pp. 267--302.

\bibitem{SS09} R. Sch\"{u}rer, W.Ch. Schmid.
MinT---new features and new results. In: P. L'Ecuyer, A.B. Owen
(eds.), {\it Monte Carlo and Quasi-Monte Carlo Methods 2008},
Springer, 2010, pp. 171--189.

\bibitem{mint} R.~Sch\"{u}rer, W.Ch.~Schmid. \textit{MinT---the database of optimal net, code, OA, and OOA parameters}.
Available at: \texttt{http://mint.sbg.ac.at} (\today).

\bibitem{SJ94} I.H.~Sloan, S.~Joe. \textit{Lattice Methods for Multiple Integration.} Clarendon Press, Oxford, 1994.
\end{thebibliography}
\end{document}